\documentclass[12pt, twoside, english, headsepline]{article}

\usepackage[margin=30mm]{geometry}

\usepackage{scrlayer-scrpage}  
\clearscrheadfoot                 
\cehead[]{F. Cinque and E. Orsingher}        
\cohead[]{Higher-order fractional equations and pseudo-processes}        
\cefoot[\pagemark]{\pagemark}  
\cofoot[\pagemark]{\pagemark}     

\usepackage{sectsty}
\sectionfont{\fontsize{14}{13}\selectfont}
\subsectionfont{\fontsize{12}{13}\selectfont}

\usepackage{amsmath,amsthm,amssymb,amsfonts, mathtools, amsbsy, dsfont}
\usepackage[utf8]{inputenx}
\usepackage{microtype}
\usepackage{hyperref}
\hypersetup{linkcolor=blue}

\newcommand*\dif{\mathop{}\!\mathrm{d}}  

\newcommand{\ceil}[1]{\left\lceil #1 \right\rceil}  
\newcommand*{\doi}[1]{\href{http://dx.doi.org/#1}{DOI: #1}}  
\newcommand*{\orcid}[1]{\href{https://orcid.org/#1}{ORCID: #1}}

\allowdisplaybreaks 

\newtheorem{theorem}{Theorem}[section] 
\newtheorem{proposition}{Proposition}[section] 

\theoremstyle{definition} 
\newtheorem{definition}{Definition}[section] 
\newtheorem{example}{Example}[section]
\newtheorem{remark}{Remark}[section] 
\numberwithin{equation}{section} 

\providecommand{\keywords}[1]
{
  \small	
  \textbf{\textit{Keywords: }} #1
}
\providecommand{\MSC}[1]
{
  \small	
  \textit{2020 MSC: } #1   
}

\title{Higher-order fractional equations and related time-changed pseudo-processes}
\author{Fabrizio Cinque$^1$ and Enzo Orsingher$^2$\\
        \small Department of Statistical Sciences, Sapienza University of Rome, Italy \\
        \small $^1$fabrizio.cinque@uniroma1.it, \orcid{0000-0002-9981-149X}\\ 
        \small$^2$enzo.orsingher@uniroma1.it
}

\begin{document}

\maketitle

\begin{abstract}
We study Cauchy problems of fractional differential equations in both space and time variables by expressing the solution in terms of ``stochastic composition" of the solutions to two simpler problems. These Cauchy sub-problems respectively concern the space and the time differential operator involved in the main equation. We provide some probabilistic and pseudo-probabilistic applications, where the solution can be interpreted as the pseudo-transition density of a time-changed pseudo-process. To extend our results to higher order time-fractional problems, we introduce stable pseudo-subordinators as well as their pseudo-inverses. Finally, we present our results in the case of more general differential operators and we interpret the results by means of a linear combination of pseudo-subordinators and their inverse processes.
\end{abstract} \hspace{10pt}

\keywords{Dzherbashyan-Caputo derivative, Pseudo-processes, Fourier transforms, Laplace transforms, Stable subordinator and inverse.}

\MSC{Primary 35R11; Secondary 60K99.}

\section{Introduction}

In this paper we study a method to solve a class of differential problems involving fractional derivatives, in the sense of Dzherbasyhan-Caputo (and more general differential operators), by considering two sub-problems, one of which with solution being interpreted, under some further hypotheses, as a pseudo-density of the generalization of the inverse of a stable subordinator.

Fractional calculus is an emerging field in mathematics which in the last decades has been intensively studied with both theoretical and applied results. Nowadays, Dzherbashyan-Caputo fractional derivatives and the associated Cauchy problems have
been investigated by many authors and, after first relevant applications in physics and mechanics (see for instance \cite{GKMR2014, P1999}), they are deeply used in all related fields of science and engineering; we refer to \cite{SZBCC2018} for a comprehensive collection of applications. We also refer to \cite{CO2024, KM2004, ZH2021} (and references therein) for some further analytic results on fractional differential equations. The Dzherbashyan-Caputo fractional derivative is defined as, for $m\in \mathbb{N}_0$ and $\alpha>-1$,
\begin{equation}\label{derivataCaputo}
\frac{\dif^\alpha}{\dif t^\alpha} f(t) = \begin{cases}
\begin{array}{l l}
\displaystyle\frac{1}{\Gamma(m-\alpha)}\int_0^t (t-s)^{m-\alpha-1}\frac{\dif^m}{\dif s^m}f(s)\dif s, & \ \text{if}\ m-1<\alpha<m,\\[9pt]
\displaystyle\frac{\dif^m}{\dif t^m} f(t), &\ \text{if} \ \alpha = m,
\end{array}\end{cases}
\end{equation}
and its Laplace transform reads, with $\mu>0$,
\begin{equation}\label{trasformataLaplaceDerivataFrazionariaIntroduzione}
\int_0^\infty e^{-\mu t}\frac{\dif^\alpha}{\dif t^\alpha}f(t) \dif t= \mu^{\alpha} \int_0^\infty e^{-\mu t}f(t) \dif t - \sum_{k=0}^{\ceil{\alpha}-1} \mu^{\alpha-k-1} \frac{\dif^{k}}{\dif t^{k}}f\Big|_{t=0}
\end{equation}
where we assume that $\lim_{t\longrightarrow \infty}  e^{-\mu t}\frac{\dif^{k}}{\dif t^{k}}f(t)=0,\ k\ge0$.
\\

We analyze problems of the form, with $\nu_1,\dots,\nu_N>0$ and $\lambda_1,\dots, \lambda_N>0$,
\begin{equation}\label{problemaIntroduzione}
\begin{cases}
\sum_{i=1}^N \lambda_i \frac{\partial^{\nu_i}}{\partial t^{\nu_i}} u(t, x) =  O_x u(t,x),\ \ t\ge0,\ x\in \mathbb{R}^d,\\[7pt]
\displaystyle\frac{\partial^{k} u}{\partial t^{k}}\Big|_{t=0} =f_k(x),\ \ x\in \mathbb{R}^d,\ k=0,\dots,\max\{\ceil{\nu_i}\,:1\le i\le N\} -1.
\end{cases}
\end{equation}
with $O_x$ being a general space differential operator which behaves ``well'' with respect to the Fourier transform, meaning that there exists a function $F$ such that  $\bigl(\mathcal{F}O_x u\bigr)(\gamma) = F(\gamma)\mathcal{F}u(\gamma)$ (see hypothesis (\ref{ipotesiOperatoreF}) for details). 

In the last section of the paper, we also consider problems of the form (\ref{problemaIntroduzione}) with more general time-operators, which behave ``well'' with respect to the Laplace transform (see hypothesis (\ref{trasformataLaplaceOperatoreTempoGenerale}) for details). 

Throughout the paper, we consider the following definition.
\begin{definition}\label{definizioneComposizioneStocastica}
Let $f:[0,\infty)\times\mathbb{R}^d \longrightarrow \mathbb{R}$ and $g:[0,\infty)\times[0,\infty) \longrightarrow \mathbb{R}$ be measureable functions, we define \textit{stochastic composition of $f$ with $g$} the function 
\begin{equation}
f\diamond g\,(t,x) = \int_0^\infty f(s,x) g(t,s)\dif s,
\end{equation}
for $t\ge0,\ x\in\mathbb{R}^d$ such that the integral converges.
\end{definition}

We point out that if $f\in L^\infty\bigl([0,\infty)\times\mathbb{R}^d\bigr)$ and $\int_0^\infty g(\cdot,x)\dif x\in L^\infty\bigl([0,\infty)\bigr) $, then $f\diamond g\in L^\infty\bigl([0,\infty)\times\mathbb{R}^d\bigr)$ and this is the kind of functions we are mostly considering along the paper.


The name ``stochastic composition'' in Definition \ref{definizioneComposizioneStocastica} justifies because if $f$ and $g$ are respectively the transition densities of two stochastic processes $\{F(t)\}_{t\ge0}$ and $\{G(t)\}_{t\ge0}$, then $f\diamond g$ is the transition density of the process $F\circ G = \{F\bigl(G(t)\bigr)\}_{t\ge0}$. Note that $F$ can be a vector-valued process. Below, we often use this interpretation of the stochastic composition of functions, applying it also for pseudo-processes.
\\

One of our main results is that the solution to problem (\ref{problemaIntroduzione}) is the stochastic composition (defined below) of the solutions to two simpler problems, one concerning the space operator (see system (\ref{problemaSpazio})) and the other one concerning the time operator (see system (\ref{problemaTempo}) for the easiest scenario and (\ref{problemaTempoGeneralizzato}) for the general framework).

Under some hypotheses concerning the initial condition of problem (\ref{problemaIntroduzione}), we are able to interpret the solution to the time problem in terms of the pseudo-density of a pseudo-inverse of a pseudo-subordinator, which extends the idea of stable-subordinators to the field of pseudo-processes (we refer to \cite{Z1986} for the theory of stable processes and stochastic subordinators). In particular, the pseudo-subordinator has density solving the fractional heat-type equation (with space variable defined on the positive half-line only),
\begin{equation}\label{equazioneCaloreFrazionaria}
\frac{\partial }{\partial t}u(t,x) = \frac{\partial^\nu}{\partial x^\nu}u(t,x),\ \ \ t,x\ge0,\ \nu>0,
\end{equation}
and its construction follows the well-known method based on the definition of the probabilistic Wiener measure and suitably adapted in the case of signed-measures (see references below).

Several authors tackled the study of processes governed by a real valued density integrating to one, meaning the so-called pseudo-probability densities: originally, their derivation and its corresponding (quasi) signed measure, was connected to higher order generalization of the heat equation, with the pioneering works of Krylov \cite{K1960}, Dalestkii and Fomin \cite{DF1965} and then Hochberg \cite{H1978}, who also introduced an Itò-type calculus for pseudo-processes for the even-order pseudo-process (that is the process with density satisfying the heat equation (\ref{equazioneCaloreFrazionaria}) with space derivative of order $2n, \ n\in\mathbb{N}$, instead of the second-order one). Then, the works of Orsingher \cite{O1992} and Lachal \cite{L2003} contributed to extend the theory to heat-type equation of all possible natural orders, meaning (\ref{equazioneCaloreFrazionaria}) with integer $\nu\ge2$. Furthermore, in these papers (see also \cite{N1997}), in view of the similarity and connections to the theory of stochastic processes, the researchers also studied some functionals of the pseudo-processes, like the first passage time and the sojourn time. We also refer to \cite{BM2015, D2006, L2007, L2012, MO2023} for further results.
\\

The paper is organized as follows. In Section 2 we study the simplest form of problem (\ref{problemaIntroduzione}), that is with $N =1$ and $\lambda_1 = 1$, i.e. considering only one fractional time derivative. This is useful to show the main results in an easier setting and introduce, in Section 3, the probabilistic and pseudo-probabilistic applications. Here, a relevant part concerns the construction of the stable pseudo-subordinator and the corresponding pseudo-inverse, also studying some of their properties in accordance to the theory of genuine probabilistic stable subordinators and their inverses. Finally, in Section 4 we state the main results in a general framework and show some further applications in the case of problems of the form (\ref{problemaIntroduzione}), focusing on the probabilistic and pseudo-probabilistic interpretations, involving linear combination of pseudo-subordinators and the related pseudo-inverse.

\section{Stochastic composition of solutions to fractional Cauchy problems}

Let $d\in \mathbb{N},\nu>0$. We denote by $\mathcal{F}u (\gamma)= \int_{\mathbb{R}^d} e^{i\gamma\cdot x} u(x)\dif x,\ \gamma \in \mathbb{R}^d$ the Fourier transform (when the integral converges) of $u:\mathbb{R}^d\longrightarrow \mathbb{R}$.

Throughout the whole paper we study differential problems under the following hypotheses. Let $O_x$ be a space-differential operator and $u:[0,\infty)\times\mathbb{R}^d\longrightarrow \mathbb{R}$ a differentiable functions such that, for natural $k$ depending on the order of the operator $O_x$,
\begin{equation}\label{ipotesiAsintoticoDerivate}
 \lim_{|x|\rightarrow\infty}\frac{\partial^k}{\partial x^k}u(t,x) = 0,\ \ \ t\ge0,
\end{equation}
where the derivative of order $k$ concerns all the needed combinations of the spacial components (according to $O_x$). Note that $u\in L^\infty$ because of the continuity of $u$ and (\ref{ipotesiAsintoticoDerivate}), which holds at least for $k=0$. Then, we assume that there exists a function $F:\mathbb{R}^d\longrightarrow \mathbb{C}$, such that 
\begin{equation}\label{ipotesiOperatoreF}
\Re\bigl(F(\gamma)\bigr)<0\ \text{ and }\ \bigl(\mathcal{F}O_x u\bigr)(\gamma) = F(\gamma)\mathcal{F}u(\gamma)\ \forall\ \gamma,
\end{equation}
where the condition on the real part is useful for the convergence of some integrals below.
\\

Now, in this section, we deal with the following Cauchy problem
\begin{equation}\label{problemaSpazioTempo}
\begin{cases}
\displaystyle \frac{\partial^{\nu } }{\partial t^{\nu }}u(t, x) =  O_x u(t,x),\ \ t\ge0,\ x\in \mathbb{R}^d,\\[10pt]
\displaystyle\frac{\partial^{k} u}{\partial t^{k}}\Big|_{t=0} = f_k(x),\ \ x\in \mathbb{R}^d,\ k=0,\dots,\ceil{\nu } -1,
\end{cases}
\end{equation}
with $f_k$ admitting Fourier transform $\forall \ k$, that we call the \textit{space-time} problem.

We relate the solution of problem (\ref{problemaSpazioTempo}) to the solutions of two simpler sub-problems. In detail, we consider the following problems:
\begin{equation}\label{problemaSpazio}
\begin{cases}
\displaystyle \frac{\partial }{\partial t}u(t, x) =  O_x u(t,x),\ \ t\ge0,\ x\in \mathbb{R}^d,\\[7pt]
\displaystyle u(0,x) = g(x),\ \ x\in\mathbb{R}^d,
\end{cases}
\end{equation}
with $g$ admitting Fourier transform, that we call the \textit{space} problem;
\begin{equation}\label{problemaTempo}
\begin{cases}
\displaystyle \frac{\partial^{\nu } }{\partial t^{\nu }}u(t, x) +\frac{\partial}{\partial x} u(t,x) = 0,\ \ t, x\ge0,\\[7pt]
\displaystyle u(t,0) = h(t), \ \ t\ge0,\\[5pt]
\displaystyle\frac{\partial^{k} u}{\partial t^{k}}\Big|_{t=0} = a_k(x),\ \ x\ge0,\ k=0,\dots,\ceil{\nu } -1,
\end{cases}
\end{equation}
with $h$ admitting Laplace transform, that we call the \textit{time} problem.
\\

Firstly, we solve the Cauchy problems (\ref{problemaSpazio}) and (\ref{problemaTempo}) and then we study their stochastic composition, say $u_1, u_2$ their corresponding solutions, we study the function
\begin{equation}\label{composizioneSoluzioniSpazioETempo}
u_1\diamond u_2(t,x) = \int_0^\infty u_1(s,x) u_2(t,s) \dif s, \ \ \ t\ge0, x\in\mathbb{R}^d.
\end{equation}
Secondly, we solve problem (\ref{problemaSpazioTempo}) and compare its solution with the function (\ref{composizioneSoluzioniSpazioETempo}).
\\

\textit{Solution to problem (\ref{problemaSpazio}).}

By using the $x$-Fourier transform, problem (\ref{problemaSpazio}) turns into
\begin{equation}\label{problemaSpazioFourier}
\begin{cases}
\displaystyle \frac{\partial }{\partial t}\mathcal{F}u(t, \gamma) =  F(\gamma)\mathcal{F}u(t, \gamma) ,\ \ t\ge0,\ \gamma\in \mathbb{R}^d,\\[7pt]
\displaystyle \mathcal{F}u(0, \gamma)  = \mathcal{F}g(\gamma) ,\ \ \gamma\in\mathbb{R}^d.
\end{cases}
\end{equation}
Then the solution is 
\begin{equation}\label{soluzioneSpazio}
\mathcal{F}u(t,\gamma)= \mathcal{F}g(\gamma)e^{tF(\gamma)},\ \ \  t\ge0,\ \gamma\in \mathbb{R}^d.
\end{equation}
\\

\textit{Solution to problem (\ref{problemaTempo}).}

By using the $t$-Laplace transform, problem (\ref{problemaTempo}) turns into (keep in mind formula (\ref{trasformataLaplaceDerivataFrazionariaIntroduzione}))
\begin{equation}\label{problemaTempoLaplace}
\begin{cases}
\displaystyle  \Bigl( \mu^\nu  \mathcal{L}u(\mu,x) - \sum_{k=0}^{\ceil{\nu}-1}\mu^{\nu-k-1} a_{k}(x) \Bigr) +\frac{\partial}{\partial x} \mathcal{L}u(\mu,x)  = 0,\ \ \mu, x\ge0,\\[10pt]
\displaystyle  \mathcal{L} u(\mu,0) = \mathcal{L}h(\mu), \ \ \mu\ge0.
\end{cases}
\end{equation}
Therefore, the solution reads 
\begin{equation}\label{soluzioneTempo}
\mathcal{L}u(\mu,x) = e^{-\mu^\nu x} \Bigl(\mathcal{L}h(\mu) + \sum_{k=0}^{\ceil{\nu}-1}\mu^{\nu-k-1} \int_0^x e^{\mu^\nu y} a_{k}(y)\dif y  \Bigr),\ \ \mu,x\ge0.
\end{equation}

Let us denote by $u_1$ the solution to problem (\ref{problemaSpazio}), whose $x$-Fourier transform is given in (\ref{soluzioneSpazio}), and by $u_2$ the solution to problem (\ref{problemaTempo}), whose $t$-Laplace transform is given in (\ref{soluzioneTempo}). The stochastic composition of the solutions $u_1,u_2$ (see formula (\ref{composizioneSoluzioniSpazioETempo})) can be expressed by means of its $t$-Laplace-$x$-Fourier transform, that is, for $\mu>0,\gamma\in\mathbb{R}^d$ (remember that $\Re\bigl(F(\gamma)\bigr)<0$),
\begin{align}
\mathcal{L}\mathcal{F}\bigl(u_1\diamond u_2\bigr)(\mu,\gamma) &=\int_0^\infty e^{-\mu t} \dif t \int_{\mathbb{R}^d} e^{i\gamma\cdot x}\dif x \int_0^\infty u_1(s,x) u_2(t,s)\dif s\nonumber\\
& = \int_0^\infty \mathcal{F}g(\gamma)e^{sF(\gamma)}\, e^{-\mu^\nu s} \Bigl(\mathcal{L}h(\mu) + \sum_{k=0}^{\ceil{\nu}-1}\mu^{\nu-k-1} \int_0^s e^{\mu^\nu y} a_{k}(y)\dif y  \Bigr) \dif s \nonumber\\
&= \mathcal{F}g(\gamma)\biggl(\frac{\mathcal{L}h(\mu)}{\mu^\nu-F(\gamma)} +  \sum_{k=0}^{\ceil{\nu}-1}\mu^{\nu-k-1}  \int_0^\infty e^{\mu^\nu y} a_{k}(y)\dif y \int_0^\infty e^{(z+y)(F(\gamma)-\mu^\nu)} \dif z\biggr)  \nonumber\\
& =\frac{ \mathcal{F}g(\gamma)}{\mu^\nu-F(\gamma)}\biggl(\mathcal{L}h(\mu) + \sum_{k=0}^{\ceil{\nu}-1}\mu^{\nu-k-1}  \int_0^\infty e^{F(\gamma) y} a_{k}(y)\dif y \biggr). \label{trasformataSoluzioneComposizione}
\end{align}

\textit{Solution to problem (\ref{problemaSpazioTempo}).}
By means of the $x$-Fourier transform and then the $t$-Laplace transform, we obtain
\begin{equation}\label{problemaSpazioTempoFourierLaplace}
\mu^\nu \mathcal{L}\mathcal{F}u(\mu,\gamma) - \sum_{k=0}^{\ceil{\nu}-1}\mu^{\nu-k-1}\mathcal{F}f_k(\gamma) = F(\gamma)\mathcal{L}\mathcal{F}u(\mu,\gamma)
\end{equation}
and therefore the solution reads
\begin{equation}\label{soluzioneProblemaSpazioTempo}
\mathcal{L}\mathcal{F}u(\mu,\gamma) = \frac{\displaystyle \sum_{k=0}^{\ceil{\nu}-1}\mu^{\nu-k-1}\mathcal{F}f_k(\gamma)}{\displaystyle \mu^\nu - F(\gamma)},\ \ \ \mu\ge0,\gamma\in\mathbb{R}^d.
\end{equation}

By comparing formulas (\ref{soluzioneProblemaSpazioTempo}) and (\ref{trasformataSoluzioneComposizione}), it is clear that we can represent the solution to problem (\ref{problemaSpazioTempo}) as stochastic composition of the solutions of the Cauchy problems (\ref{problemaSpazio}) and (\ref{problemaTempo}) with suitable initial conditions. In detail, it is easy to show that the following conditions are sufficient for such representation:
\begin{itemize}
\item[($a$)] fix $j\in \{0,\dots, \ceil{\nu}-1\}$, $f_j \not = 0$ and $f_k =0,\ k\not =j$, then it is sufficient that $g = f_j, a_k = 0\ \forall\ k$ and $h$ such that $\mathcal{L}h(\mu) = \mu^{\nu-j-1}$,
\item[($b$)] for general $f_k\ \forall\ k$, it is sufficient that $h=0,\, g = \delta$ (i.e. $\mathcal{F}g(\gamma) = 1\ \forall\ \gamma)$ and $a_k$ such that $\int_0^\infty e^{yF(\gamma)}a_k(y)\dif y = \mathcal{F}f_k(\gamma)\ \forall\ \gamma,k$,
\end{itemize}
where, as usual, $\delta$ denotes the Dirac delta function centered in the origin.
\\
Note that in point ($b$), the last condition can also be expressed as $\mathcal{L}a_k\bigl(-F(\gamma)\bigr) = \mathcal{F}f_k(\gamma)$.

The above arguments prove the following theorem.

\begin{theorem}\label{teoremaComposizioneStocasticaSoluzioniSpazioTempo}
The solution to the differential Cauchy problem (\ref{problemaSpazioTempo}), under the hypotheses (\ref{ipotesiOperatoreF}), can be expressed as the stochastic composition of the solution to (the space) problem (\ref{problemaSpazio}) with the solution to (the time) problem (\ref{problemaTempo}), where these problems have initial conditions given in point ($b$) above.
\end{theorem}

Let us now consider the limit behavior for $\nu\,\downarrow\,0$ of problem (\ref{problemaSpazioTempo}) and its solution. It is interesting to observe that under some hypothesis (required to exchange limit and integral), the limit of the solution to problem (\ref{problemaSpazioTempo}) is solution to the limit of such differential problem.

\begin{proposition}\label{proposizioneLimiteAZero}
Let us denote by $u_\nu$ the solution to problem (\ref{problemaSpazioTempo}). Assume that there exists the pointwise limit $u_0=\lim_{\nu\,\downarrow\, 0} u_\nu$ and that $\mathcal{L}\mathcal{F} u_0 = \lim_{\nu\,\downarrow\, 0} \mathcal{L}\mathcal{F}u_\nu$. Then, $u_0$ satisfies the problem, for $t\ge0$,
\begin{equation}\label{problemaLimiteSpazioTempo}
\begin{cases}
u(t,x) - f_0(x) =  O_x u(t,x),\ \ \ x\in \mathbb{R}^d,\\[0pt]
\displaystyle\lim_{|x|\longrightarrow\infty} u(t,x) = 0.
\end{cases}
\end{equation}
that is the ordinary Cauchy problem obtained as the limit of problem (\ref{problemaSpazioTempo}).
\end{proposition}

Clearly, a similar (but less interesting) result equivalently holds in the case of the limit $\nu\longrightarrow \nu_0>0$, see also Remark 3.2 of \cite{CO2024} (where the authors work after the application of the Fourier transform, which reduces the equation to an ordinary fractional differential equation in the variable $t$).

\begin{proof}
Assume suitable $u$ such that $\lim_{\nu\,\downarrow\, 0} \int_0^t (t-s)^{-\nu} u'(s)\dif s = \int_0^t \lim_{\nu\,\downarrow\, 0} (t-s)^{-\nu} u'(s)\dif s $ (for instance $u$ such that $u'\in L^\infty\bigl([0,\infty)\bigr)$), then
\begin{equation}\label{limiteDerivataCaputo}
\lim_{\nu\,\downarrow\, 0} \frac{\partial^{\nu } }{\partial t^{\nu }}u(t, x) = \int_0^t \frac{\partial }{\partial s}u(s,x)\dif s = u(t,x) - u(0,x),\ \ \ t\ge0,x\in \mathbb{R}^d,
\end{equation}
where we used the definition (\ref{derivataCaputo}) of fractional derivative.

Thus, by applying the limit (\ref{limiteDerivataCaputo}), problem (\ref{problemaSpazioTempo}) coincides with (\ref{problemaLimiteSpazioTempo}), where in the latter we may omit the variable $t$ because there are just $x$-differential operators.

Now, by keeping in mind the $t$-Laplace $x$-Fourier transform (\ref{soluzioneProblemaSpazioTempo}), we can write
$$  \mathcal{L}\mathcal{F} u_0(\mu,\gamma) = \lim_{\nu\,\downarrow\, 0} \mathcal{L}\mathcal{F}u_\nu (\mu,\gamma)= \frac{1}{\mu}\frac{\mathcal{F}f_0(\gamma)}{1-F(\gamma)}, \ \ \ \mu>0,\gamma\in \mathbb{R}^d,$$
whose $t$-Laplace inverse reads
\begin{equation}\label{trasformataFourierLimiteSoluzione}
\mathcal{F} u_0(t,\gamma) = \frac{\mathcal{F}f_0(\gamma)}{1-F(\gamma)}, \ \ \ t\ge0,x\in \mathbb{R}^d,
\end{equation}
that is constant in $t$. Finally, it is easy to see that (\ref{trasformataFourierLimiteSoluzione}) is the Fourier transform of the solution to the differential problem (\ref{problemaLimiteSpazioTempo}).
\end{proof}

\section{Probabilistic and pseudo-probabilistic applications}

In this section we mainly focus on the Cauchy problem (\ref{problemaSpazioTempo}) with null initial conditions except for the first one, that is 
\begin{equation}\label{problemaSpazioTempoProbabilita}
\begin{cases}
\displaystyle \frac{\partial^{\nu } }{\partial t^{\nu }}u(t, x) =  O_x u(t,x),\ \ t\ge0,\ x\in \mathbb{R}^d,\\[7pt]
u(0,x) = \delta(x),\ \  x\in \mathbb{R}^d,\\[5pt]
\displaystyle\frac{\partial^{k} u}{\partial t^{k}}\Big|_{t=0} =0,\ \ x\in \mathbb{R}^d,\ k=1,\dots,\ceil{\nu } -1.
\end{cases}
\end{equation}

In light of Theorem \ref{teoremaComposizioneStocasticaSoluzioniSpazioTempo} we can interpret the solution to (\ref{problemaSpazioTempoProbabilita}) in terms of subordination with the inverse of a stable subordinator if $\nu\in(0,1]$ and its generalization otherwise. In fact, by keeping in mind point ($a$) before Theorem \ref{teoremaComposizioneStocasticaSoluzioniSpazioTempo}, the \textit{time} problem related to system (\ref{problemaSpazioTempoProbabilita}) is given by 
\begin{equation}\label{problemaTempoProbabilita}
\begin{cases}
\displaystyle \frac{\partial^{\nu } }{\partial t^{\nu }}u(t, x) =  -\frac{\partial }{\partial x}u(t,x),\ \ t,x\ge0,\\[5pt]
 u(t,0) = \mathcal{L}^{-1}\bigl(\mu^{\nu-1}\bigr)(t),\ \ t\ge,0\\[4pt]
\displaystyle\frac{\partial^{k} u}{\partial t^{k}}\Big|_{t=0} =0,\ \ x\ge0,\ k=0,\dots,\ceil{\nu } -1,
\end{cases}
\end{equation}
whose solution has $t$-Laplace transform equal to $\mathcal{L}u(\mu,x) = \mu^{\nu-1 }e^{-\mu^\nu x}$, see (\ref{soluzioneTempo}).
\\

In the case of $\nu\in (0,1]$ the first boundary condition of (\ref{problemaTempoProbabilita}) becomes $u(t,0) = t^{-\nu}/\Gamma(1-\nu)$ since $ \int_0^\infty e^{-\mu t} t^{-\nu}\dif t/ \Gamma(1-\nu) = \mu^{\nu-1}$. Then, it is well-known (see \cite{GKMR2014, Z1986}) that the solution of the \textit{time} problem is the probability transition density of the inverse of a stable subordinator of order $\nu$, $L_\nu$, i.e. the function
\begin{equation}\label{leggeInversoSubordinatore} 
l_\nu(t,x) = \frac{1}{t^\nu}W_{-\nu, 1-\nu}\Bigl(-\frac{x}{t^\nu}\Bigr),\ \ \ t,x\ge0,
\end{equation}
where $W_{\alpha,\beta}(z) = \sum_{k=0}^\infty z^k/\bigl(k!\,\Gamma(\alpha k+\beta)\bigr), $ with $\alpha>-1,\beta,z\in \mathbb{C}$, is the Wright function.
\\

Before generalizing the notion of inverse of stable subordinators in the case of a non-probabilistic ``measure'' and giving further results concerning problem (\ref{problemaSpazioTempoProbabilita}), we give a probabilistic example of Proposition \ref{proposizioneLimiteAZero}. From the results above, we need to study the limit of the inverse of the stable subordinator for $\nu\,\downarrow\, 0$. By considering the $t$-Laplace transform it is easy to show that  $L_\nu(t) \stackrel{d}{\longrightarrow} Y\sim Exp(1)$, so the limit is an exponential random variable not depending on time $t$. In fact, with $ t,x\ge0$,
$$\lim_{\nu\,\downarrow\, 0} l_\nu(t,x) = \mathcal{L}_\mu^{-1} \mathcal{L}_t \Bigl(\lim_{\nu\,\downarrow\, 0} l_\nu(t,x)\Bigr) = \mathcal{L}_\mu^{-1} \Bigl(\lim_{\nu\,\downarrow\, 0} \mu^{\nu-1}e^{-\mu^\nu x}\Bigr) = \mathcal{L}_\mu^{-1}\Bigl( \frac{e^{-x}}{\mu} \Bigr)= e^{-x}.$$
One can also prove this result by means of the explicit form of the transition probability density (\ref{leggeInversoSubordinatore}).

\subsection{General theory of pseudo-processes}\label{sezioneTeoriaGeneralePseudoProcessi}

Let $E \subseteq \mathbb{R}$ and $u:[0,\infty)\times E \longrightarrow \mathbb{R}$ be a continuous function such that for all $t\ge0$, $\int_E u(t,x)\dif x = 1$ and $\int_E|u(t,x)| \dif x =+\infty$. We now show how to define and  construct a pseudo-process $X = \{X(t)\}_{t\ge0}$ taking values in $E$ and having pseudo-density (or kernel or, for brevity, density) $u$. Our construction is inspired by previous remarkable papers, see \cite{H1978, K1960}, where the authors applied a method similar to the Wiener's procedure used for Brownian motion.
\\

Let $\Omega=\{g:[0,\infty)\longrightarrow E\,:\, g\ \text{measurable}\}$ called the \textit{sample path space}. By means of $u$ we construct a set function $Q$ on the functional space $\Omega$. 
\\Let us consider the cylinders in $\Omega$, $C_n(\underline{I},\underline{t})=\{g\in\Omega\,:\,g(t_i)\in I_i, \ i=1,\dots n\}$, with $n\in\mathbb{N}$, $\underline{I}$ is a collection of Borel sets of $E$, $I_1,\dots,I_n$ (i.e. $I_i\in \mathcal{B}(E)\ \forall\ i$) and $\underline{t}$ is a vector of time instants $0\le t_1<\dots<t_n<\infty$. We define the function $Q$ on a cylinder $C_n(\underline{I},\underline{t})$ as follows
\begin{equation}\label{definizionePseudoProbabilita}
Q\bigl(C_n(\underline{I},\underline{t})\bigr) = \int_{I_1}\cdots\int_{I_n} \prod_{i=1}^n u(t_i-t_{i-1},x_i-x_{i-1}) \dif x_i,
\end{equation}
where $t_0,x_0=0$.

From the hypotheses on $u$, we readily obtain that $Q$ is a set function with infinite total variation and it is not a signed-measure, thus the classical measure theory does not apply (and we cannot use Kolmogorov extension theorem because its hypotheses are not satisfied). In order to build a feasible framework for the set function $Q$ we proceed as follows.

First, we extend the set function on the class of all the cylinders $ \{C_n(\underline{I},\underline{t})\,:\, n\in \mathbb{N}, \, \underline{I}\in\mathcal{B}(E)^n,\,\underline{t}\in[0,\infty)^n\ \text{s.t.}\ 0\le t_1<\dots<t_n\}$. Then, in order to avoid the indeterminate form $+\infty-\infty$ and assure the additivity of $Q$, we extend it to the algebra of the cylinders with finite integral, that is
\begin{equation}\label{algebra}
\mathcal{A} = \alpha\Bigl(  \big\{C_n(\underline{I},\underline{t})\,:\, n\in \mathbb{N}, \, \underline{I}\in\mathcal{B}(E)^n,\,\underline{t}\in[0,\infty)^n\ \text{s.t.}\ 0\le t_1<\dots<t_n, \, \big|Q\bigl(C_n(\underline{I},\underline{t})\bigr)\big|<\infty \big\} \Bigr),
\end{equation}
where $\alpha$ denotes the smallest algebra containing the input family (more precisely, $\alpha(\mathcal{F})$ is the intersection of all the algebras containing the family $\mathcal{F}$). 

We call $Q$ the \textit{pseudo-measure with pseudo-density $u$} (with respect to the Lebesgue measure) and the elements of $\mathcal{A}$ \textit{$Q$-measurable} sets. We say that a subset $E\in \mathcal{A}$ is \textit{$Q$-negligible} if $Q(F)=0$ for all $Q$-measurable subsets $F\subseteq E$, i.e. a set is negligible if it has value $0$ and the same holds for all its subsets. We also say that a property is satisfied \textit{$Q$-almost everywhere} ($Q$-a.e.) if it does not hold at most on a negligible set.
\\

Now, we consider the sequence of algebras $\{\mathcal{A}_t\}_{t\ge0}$ where, for $s>0$,
\begin{equation}\label{algebra}
\mathcal{A}_s = \alpha\Bigl(  \big\{C_n(\underline{I},\underline{t})\,:\, n\in \mathbb{N}, \, \underline{I}\in\mathcal{B}(E)^n,\,\underline{t}\in[0,s]^n\ \text{s.t.}\ 0\le t_1<\dots<t_n, \, \big|Q\bigl(C_n(\underline{I},\underline{t})\bigr)\big|<\infty \big\} \Bigr)
\end{equation}
which is an increasing collection of algebras such that $\mathcal{A}_s\uparrow \mathcal{A}$.

Eventually, the pseudo-process $X$ with pseudo-density (or kernel) $u$ is well defined on the pseudo-measurable filtered space $\bigl(\Omega, \mathcal{A}, \{\mathcal{A}_t\}_{t\ge0}, Q\bigr)$ and we write $ Q\bigl(X(t_1)\in I_1,\dots , X(t_n)\in I_n\bigr) = Q\bigl(C_n(\underline{I},\underline{t})\bigr)$.

For $t\ge0$, it is useful to clarify that the element $X(t)$ of the pseudo-process has pseudo-density $u_t (\cdot)= u(t,\cdot)$ which leads to a pseudo-measure $Q_t$ defined on the algebra of the cylinders which keep into account only the time instant $t$, that is 
\begin{equation} \label{algebraPseudoVariabile}
\mathcal{C}_t =\alpha\Bigl( \big\{C(I,t)\,:\, I\in \mathcal{B}(E)\big\} \Bigr).
\end{equation}
Note that $Q_t(I) = Q\bigl(C(I,t)\bigr)$ and therefore it follows that $Q$ restricted to $\mathcal{C}_t$ coincides with $Q_t$. Thus, we can define the pseudo-random variable $X(t)$ on the pseudo-measurable space $(\Omega, \mathcal{C}_t, Q)$.

We now state some definitions about the pseudo-processes which resembles the conventional stochastic statements.

\begin{definition}\label{definizioniPsudoProcessi}
Let $X$ be a pseudo-process with pseudo-density $u:[0,\infty)\times E \longrightarrow \mathbb{R}$, defined on the pseudo-measurable filtered space $\bigl(\Omega, \mathcal{A}, \{\mathcal{A}_t\}_{t\ge0}, Q\bigr)$.
\begin{itemize}
\item[($i$)] Assume $E\subseteq [0,\infty)$. For $t\ge0$, the \textit{moment generating function of $X(t)$} is the $x$-Laplace transform of $u(t, \cdot)$. The \textit{moment generating function of $X$} is the $x$-Laplace transform of $u$.

\item[($ii$)] For $t\ge0$, the \textit{characteristic function of $X(t)$} is the $x$-Fourier transform of $u(t, \cdot)$. The \textit{characteristic function of $X$} is the $x$-Fourier transform of $u$. 
\end{itemize}
Let $X'$ be a pseudo-process with pseudo-density $v:[0,\infty)\times E' \longrightarrow \mathbb{R}$, defined on the filtered space $\bigl(\Omega', \mathcal{A}', \{\mathcal{A}'_t\}_{t\ge0}, Q'\bigr)$.

\begin{itemize}
\item[($iii$)] For $s,t \ge0$, we say that $X(t)$ and $X'(s)$ are \textit{equal in distribution} if $\mathcal{C}_t =\mathcal{C}'_s$ and $Q\bigl(X(t)\in A\bigr) = Q'\bigl(X'(s)\in A\bigr)$ for all $A$; we write $X(t)\stackrel{d}{=}X'(s)$.

\item[($iv$)] $X$ has \textit{stationary increments} if $X(t)-X(s) \stackrel{d}{=}X(t-s)$ for $t\ge s \ge 0$.
\end{itemize}
Let $Q\times Q'$ being a pseudo-measure on the algebra $\alpha\bigl(\mathcal{A}\times\mathcal{A}'\bigr)$. We define the joint pseudo-measurable space $\Bigl(\Omega\times\Omega',\alpha\bigl(\mathcal{A}\times\mathcal{A}'\bigr), \big\{\alpha\bigl(\mathcal{A}_t\times\mathcal{A}'_t\bigr)\big\}_{t\ge0}, Q\times Q'\Bigr)$. 

\begin{itemize}
\item[($v$)] For $s,t\ge0$, we say that $X(t)$ and $X'(s)$ are \textit{independent} if $Q\times Q'(A\times A') = Q(A)Q'(A')$ for all $A\times A'\in\alpha\bigl(\mathcal{C}_t\times\mathcal{C}'_s\bigr)$.

\item[($vi$)] $X$ and $X'$ are \textit{independent} if $Q\times Q'(A\times A') = Q(A)Q'(A')$ for all $A\times A'\in\alpha\bigl(\mathcal{A}\times\mathcal{A}'\bigr)$.

\item[($vii$)] Assume $X,X'$ independent and $X'\ge0$ $Q'$-a.e.. We define the \textit{time-changed} pseudo-process $X\circ X' = \bigl\{X\bigl(X'(t)\bigr)\}_{t\ge0}$ as the process having pseudo-density $u\diamond v$ (see Definition \ref{definizioneComposizioneStocastica}).
\end{itemize}
\end{definition}

We point out that definitions ($i$), ($ii$) and ($vii$) make sense only if the corresponding integrals converge.
\\
In definitions ($iii$) and ($v$) we consider the pseudo-measurable space of the components $X(t)$ and $X'(s)$, having algebras of the type (\ref{algebraPseudoVariabile}). 
\\
It is important to observe that definition ($iii$) can be also expressed in terms of the equality of the characteristic functions (or moment generating functions).
\\
In light of the independence definition ($v$), the pseudo-density of the pseudo-measure $Q\times Q'$ is the product of the pseudo-densities $u(t,\cdot), v(s,\cdot)$ (also known as marginal pseudo-densities) and therefore also the joint characteristic function (or moment generating function) is the product of the marginal ones. 
Similar factorizations hold in definition ($vi$), considering the product of $u$ and $v$ as well as of their characteristic functions (or moment generating functions).
\\

Note that the definitions of independence and time-changed composition, i.e. ($v$)-($vi$)-($vii$), concerning the joint space, can be extended to the case where $X'$ (or $X$) is a conventional stochastic process. Indeed, $\mathcal{A}',\ \mathcal{A}'_t$ and $\mathcal{C}'_t, \ t\ge0,$ would be $\sigma$-algebras on $\Omega'$ with $Q'$ a probability measure. In general, the joint space would still be a pseudo-measurable space with a pseudo-measure $Q\times Q'$ defined on an algebra and whose kernel is expressed as the product of the pseudo-density of $X$ and the probability law of $X'$. In the case of the time-changed process, the pseudo-density would be defined as the stochastic composition (see Definition \ref{definizioneComposizioneStocastica}) of a kernel and a probability density.

Below we deal with both pseudo-processes and canonical stochastic processes and, by assuming independence, we work using the factorization of the moment generating functions.

\subsection{Stable pseudo-subordinators and pseudo-inverses}

Let $\nu>1$. We consider the fractional Cauchy problem
\begin{equation}\label{problemaFrazionarioPseudoSubordinatore}
\begin{cases}
\displaystyle \frac{\partial}{\partial t}u(t, x) = - \frac{\partial^{\nu } }{\partial x^{\nu }}u(t,x),\ \ t,x\ge0,\\[8pt]
u(0,x) = \delta(x),\ \ x\ge0,\\[6pt]
\displaystyle\frac{\partial^{k} u}{\partial x^{k}}\Big|_{x=0} = 0,\ \ t\ge0,\ k=0,\dots,\ceil{\nu } -1,
\end{cases}
\end{equation}
whose solution is denoted by $u_\nu$ (belonging to $L^\infty$) and it has Laplace transform given by
\begin{equation}\label{LaplacePseudoSubordinatore}
\mathcal{L}u_\nu(t,\mu) = e^{-\mu^\nu t},\ \ \ t,\mu\ge0.
\end{equation}
With (\ref{LaplacePseudoSubordinatore}) at hand we readily have that $\int_0^\infty u_\nu(t,x)\dif x = 1\ \forall\ t$ and that $u$ is not a probability density since it takes real values. Indeed (remember $\nu>1$),
$$\int_0^\infty x u_\nu(t,x) \dif x = -\frac{ \partial}{\partial \mu} \mathcal{L}u_\nu(t,\mu)\Big|_{\mu = 0} = \nu t\mu^{\nu-1} e^{-\mu^\nu t}\Big|_{\mu=0} = 0$$
meaning that $u_\nu$ must be sign-varying.

We point out that if $\nu\in(0,1]$ the solution to (\ref{problemaFrazionarioPseudoSubordinatore}) is the probability law of a stable subordinator of order $\nu$; formula (\ref{LaplacePseudoSubordinatore}) is the moment generating function and we can also write $u_\nu(t,x) = \nu t\, l_\nu(x,t) /x $, with $l_\nu$ density of the inverse, see (\ref{leggeInversoSubordinatore}).
\\

By following the procedure presented in Section \ref{sezioneTeoriaGeneralePseudoProcessi}, setting $E=[0,\infty)$, we construct the pseudo-process $S_\nu = \{S_\nu(t)\}_{t\ge0}$ with kernel $u_\nu$, i.e. such that $Q\{S_\nu(t_1)\in I_1,\dots,S_\nu(t_n)\in I_n\} = Q\bigl(C(\underline{I},\underline{t})\bigr)$, suitably adapted from (\ref{definizionePseudoProbabilita}), for $I_1,\dots,I_n\in \mathcal{B}\bigl([0,\infty)\bigr)$ and $0\le t_1<\dots<t_n<\infty$. We call $S_\nu$ \textit{stable pseudo-subordinator of order $\nu>1$}. It starts at $S_\nu(0) = 0$ $Q$-a.e. because of the first initial condition of problem (\ref{problemaFrazionarioPseudoSubordinatore}), and from the Laplace transform (\ref{LaplacePseudoSubordinatore}) (which, according to ($i$) of Definition \ref{definizioniPsudoProcessi}, is its moment generating function) we derive that it has independent and stationary increments since we can decompose $e^{-\mu^\nu t} = e^{-\mu^\nu (t-s)}e^{-\mu^\nu s}$.
The stable pseudo-subordinator assumes non-negative values $Q$-a.e. because of the definition of the pseudo-measure (\ref{definizionePseudoProbabilita}) and since $u_\nu(t,x)$ is defined only for $x\ge0$ (we may extend it to $\mathbb{R}$ by setting $u(t,x) = 0$ for $x<0, t\ge0$). 
\\Hence, $S_\nu$ is a non-negative and non-decreasing pseudo-process with moment generating function (\ref{LaplacePseudoSubordinatore}). These reasons justify the name ``subordinator'' in the definition.

\begin{remark}[Properties of stable pseudo-subordinators]
As for Brownian motion and other pseudo-processes (see for instance \cite{H1978}), thanks to formula (\ref{LaplacePseudoSubordinatore}), it is straightforward to obtain the following properties of stable pseudo-subordinators. For $\nu>0$ (i.e. including also the probabilistic case) and $t\ge0$,
\begin{itemize}
\item[($i$)] the $t$-Laplace transform of $u_\nu$ can be derived as follows, with $0\le \delta<\mu^\nu$:
\begin{align*} 
\int_0^\infty e^{-\delta t} \int_0^\infty e^{-\mu x}u_\nu(t,x)&\dif x\dif t=\int_0^\infty e^{-\delta t} e^{-\mu^\nu t}\dif t  = \frac{1}{\mu^\nu+\delta} \nonumber\\&\implies \int_0^\infty e^{-\delta t} u_\nu(t,x)\dif t = x^{\nu-1}E_{\nu,\nu}(-\delta x^\nu),\ \ x\ge0,
\end{align*}
where $E_{\alpha,\beta}(z)  =\sum_{k=0}^\infty z^k/\Gamma(\alpha k +\beta)$, with $\alpha,\beta,z\in\mathbb{C}$ and $\Re(\alpha)>0$, is the Mittag-Leffler function;
\item[($ii$)] $S_\nu(t)\stackrel{d}{=}t^{2/\nu} S_\nu(1/t) \stackrel{d}{=} c^{1/\nu} S_\nu(t/c)$, with $c>0$;
\item[($iii$)] let $S_\nu^{(1)},S_\nu^{(2)}\dots$ be independent pseudo-subordinators and $\{a_n\}_{n\in\mathbb{N}}$ be a non-negative real sequence such that $\sum_{k=1}^n a_k^\nu/n \longrightarrow a<\infty$, then 
\begin{equation}\label{vesioneLimiteCentrale}
\frac{1}{n^{1/\nu}}\sum_{k=1}^n a_k S_\nu^{(k)}(t) \stackrel{d}{\longrightarrow} a^{1/\nu} S_\nu(t) \stackrel{d}{=}S_\nu(at).
\end{equation}
In fact, thanks to the independence assumption and that from ($ii$) $c S_\nu(t) = S_\nu(c^\nu t)$, with $c>0$, the $x$-Laplace transform of the pseudo-density of the first member of (\ref{vesioneLimiteCentrale}) can be written as,
$$ \prod_{k=1}^n \int_0^\infty e^{-\mu x} u_{\nu}\Bigl(\frac{a_k^\nu t}{n}, x\Bigr)\dif t = \prod_{k=1}^n e^{-\mu^\nu a_k^\nu t/n} \xrightarrow[n\rightarrow \infty]{} e^{-\mu^\nu a t }.$$
\item[($iv$)] with $\nu_1,\nu_2>0$, $S_{\nu_1}	\bigl(S_{\nu_2}(t)\bigr) \stackrel{d}{=} S_{\nu_1\nu_2}(t)$.
\end{itemize}
It is well-known that these properties hold in the case of $\nu\in(0,1)$. The latter property can be linked to the idea of stable pseudo-random variables (and processes).  
\hfill$\diamond$
\end{remark}

In analogy to the derivation of probabilistic subordinators, we derive the representation of a pseudo-subordinator as limit of a compound Poisson with ingredients which include genuine random variables with power-law distribution and suitable sign-varying terms (which are the pseudo-subordinators themselves at time $t=1$).

\begin{proposition}\label{proposizioneLimitePoissonCompostoPseudoSub}
Let $S_\nu^{(1)},S_\nu^{(2)},\dots$ be independent stable pseudo-subordinators of order $\nu>1$ and $N$ be an independent homogeneous Poisson process of rate $\lambda >0$. Then
\begin{equation}\label{limitePoissonCompostoPseudoSubBeta}
\lim_{\delta\downarrow0} \sum_{k=1}^{N(t\,\delta^{-\beta})} X_k^{\beta,\delta}S_\nu^{(k)}(1) \stackrel{d}{=} S_\beta\bigl(\lambda t \,\Gamma(1-\beta/\nu)\bigr),
\end{equation}
where $0<\beta<\nu$ and $X_1^{\beta,\delta},X_2^{\beta,\delta},\dots$ are independent copies of the random variable $X^{\beta, \delta}$ such that
\begin{equation}\label{definizioneVariabiileSopravvivenzaPotenza}
P\{X^{\beta,\delta}>x\} =
\begin{cases}
\begin{array}{l l}
1, & x<\delta,\\
\Bigr(\frac{\delta}{x}\Bigr)^\beta, & x\ge\delta.
\end{array}
\end{cases}
\end{equation}
Furthermore, we can write
\begin{equation}\label{limitePoissonCompostoPseudoSub}
\lim_{\delta\downarrow0}  \sum_{k=1}^{N(t\,\delta^{-\nu})} \delta S_\nu^{(k)}(1) \stackrel{d}{=}  S_\nu(\lambda t).
\end{equation}
\end{proposition}

The interested reader can find some further results of the type of those in Proposition \ref{proposizioneLimitePoissonCompostoPseudoSub} in the paper of Orsingher and Toaldo \cite{OT2014} (concerning other classes of pseudo-processes).

\begin{proof}
We study the pseudo-random object
\begin{equation}\label{sommaPoissonCompostoPseudoProcessoGenerale}
\sum_{k=1}^{N(t\,\delta^{-\alpha})} X_k^{\beta,\delta}S_\nu^{(k)}(1),\ \  \text{for }\ \alpha,\beta,\nu,\delta>0. 
\end{equation}
Let us denote by $f(x) = \beta \delta^\beta/x^{\beta+1}\mathds{1}_{[\delta,\infty)}(x),\ x\in\mathbb{R},$ the density of the random variable in (\ref{definizioneVariabiileSopravvivenzaPotenza}). Note that all the subordinators have the same pseudo-distribution $u(1,\cdot)$ (omitting the subscript $\nu$) and keep in mind the independence among all the elements appearing in (\ref{sommaPoissonCompostoPseudoProcessoGenerale}). Now, the Laplace transform of the pseudo-density of (\ref{sommaPoissonCompostoPseudoProcessoGenerale}) writes as
\begin{align}
\sum_{n=0}^\infty P\{N(t\delta^{-\alpha}) = n\} &\int_\delta^\infty f(x_1)\dif x_1\cdots\int_\delta^\infty f(x_n)\dif x_n \int_0^\infty u(1, y_1)\dif y_1 \cdots\int_0^\infty u(1, y_n)\dif y_n e^{-\mu\sum_{k=1}^n x_ky_k}\nonumber\\
& = e^{-\lambda t\delta^{-\alpha}}\sum_{n=0}^\infty  \frac{\Bigl(\lambda t\delta^{-\alpha}\Bigr)^n}{n!}\prod_{k=1}^n \Biggl( \int_\delta^\infty f(x_k) \dif x_k \int_0^\infty e^{-\mu x_k y_k} u(1,y_k)\dif y_k  \Biggr) \nonumber\\
& = \exp\Bigg\{ -\lambda t\delta^{-\alpha} \Biggl(1- \int_\delta^\infty f(x) \dif x \int_0^\infty e^{-\mu x y} u(1,y)\dif y \Biggr) \Bigg\} \nonumber\\
& = \exp\Bigg\{ -\lambda t\delta^{-\alpha} \Biggl(1- \int_\delta^\infty f(x) e^{-(\mu x)^\nu}\dif x\Biggr)  \Bigg\}\nonumber\\
& = \exp\Bigg\{ -\lambda t\delta^{-\alpha} \int_\delta^\infty \Bigl(1- e^{-(\mu x)^\nu} \Bigr)f(x)\dif x  \Bigg\}\nonumber\\
&  = \exp\Bigg\{ -\lambda t\beta\delta^{\beta-\alpha} \int_\delta^\infty \frac{1- e^{-\mu^\nu x^\nu} }{x^{\beta+1}}\dif x  \Bigg\}\nonumber\\ 
& = \exp\Bigg\{ -\lambda t\delta^{-\alpha}\Bigl(1- e^{-\mu^\nu \delta^\nu}\Bigr)-\lambda t\nu\mu^\nu \delta^{\beta-\alpha}  \int_\delta^\infty e^{-\mu^\nu x^\nu} x^{\nu-\beta-1}\dif x  \Bigg\}\nonumber\\
& \stackrel{\delta\downarrow 0}{\longrightarrow}\begin{cases}
\begin{array}{l l}
1,& \alpha<\beta,\nu,\\
e^{-\lambda t\mu^\beta \Gamma(1-\beta/\nu)}, & \alpha = \beta<\nu,\\
e^{-\lambda t \mu^\nu}, &\alpha = \nu<\beta,\\
0, & \beta\le\alpha\le\nu \text{ or } \alpha>\nu.
\end{array}
\end{cases}\label{dimostrazionePoissonCompostoSubordinatore}
\end{align}
We observe that in the second last step we used integration by parts.
\\
Finally, the second and third cases of (\ref{dimostrazionePoissonCompostoSubordinatore}) respectively yields (\ref{limitePoissonCompostoPseudoSubBeta}) and (\ref{limitePoissonCompostoPseudoSub}); note that $X^{\beta,\delta}\stackrel{d}{\longrightarrow} X^\delta = \delta$ a.s. for $\beta\longrightarrow\infty$.
\end{proof}

We point out that the formulas in Proposition \ref{proposizioneLimitePoissonCompostoPseudoSub} prove that $S_\nu$ has non-decreasing piecewise constant trajectories.
\\

Since $S_\nu$ is a non decreasing pseudo-process taking values in $[0,\infty)$, following the line of the classical theory of probabilistic subordinators, we define the \textit{pseudo-inverse} (of the stable pseudo-subordinator of order $\nu>1$) $L_\nu = \{L_\nu(t)\}_{t\ge0}$ such that $L_\nu(t) = \inf\{x\ge0\,:\,S_\nu(x)\ge t\}$. By working as for classical subordinators and inverses, we obtain that, denoting with $l_\nu:[0,\infty)\times[0,\infty)\longrightarrow \mathbb{R}$ the pseudo-density of $L_\nu$, then, with $\mu,t, x\ge0$,
\begin{equation}\label{relazioniPseudoSunordinatoreEdInverso}
\int_0^x l_\nu(t,y) \dif y = \int_t^\infty u_\nu(x,y) \dif y \ \text{ and }\  \int_0^\infty e^{-\mu t}l_\nu(t,x)\dif t = - \int_0^\infty e^{-\mu t} \dif t \frac{\partial}{\partial x} \int_0^t u_\nu(x,s)\dif s, 
\end{equation}
which lead to
\begin{equation}\label{LaplaceInversoPseudoSubordinatore}
\int_0^\infty e^{-\mu t} l_\nu(t,x) \dif t = \mu^{\nu-1}e^{-\mu^\nu x},\ \ \ \mu,x\ge0.
\end{equation}
Note that when we use $l_\nu(t,x)$, the first argument represents the ``time'' for $L_\nu$, meaning the ``space'' for $S_\nu$ and the other way round for the second argument. Thus, formula (\ref{LaplaceInversoPseudoSubordinatore}) is not the moment generating function of $L_\nu(t)$. However, we can derive this function by considering the $(x,t)$-Laplace transforms and ``inverting'', that is, with $\delta\ge0$,
\begin{align*}
\int_0^\infty e^{-\delta x}\mu^{\nu-1}e^{-\mu^\nu x}& \dif x = \frac{\mu^{\nu-1}}{\mu^\nu+\delta}, \ \ \mu\ge0\implies \int_0^\infty e^{-\delta x }l_\nu(t,x) \dif x = E_{\nu,1} (-\delta t^\nu),\ \ t\ge0. 
\end{align*}

In view of formula (\ref{LaplaceInversoPseudoSubordinatore}), we have that the pseudo-density of $L_\nu$ is the solution to the Cauchy problem (\ref{problemaTempoProbabilita}).

Hence, the task of solving problem (\ref{problemaSpazioTempoProbabilita}) reduces to study the associated \textit{space} problem
\begin{equation}\label{problemaSpazioPseudoProbabilita}
\begin{cases}
\displaystyle \frac{\partial }{\partial t}u(t, x) =  O_x u(t,x),\ \ t\ge0,\ x\in \mathbb{R}^d,\\[7pt]
\displaystyle u(0,x) = \delta(x),\ \ x\in\mathbb{R}^d,
\end{cases}
\end{equation}
In light of Theorem \ref{teoremaComposizioneStocasticaSoluzioniSpazioTempo}, the solution to problem (\ref{problemaSpazioTempoProbabilita}) is given by the stochastic composition of the solution to (\ref{problemaSpazioPseudoProbabilita}) and the kernel of a pseudo-inverse, $l_\nu$.
\\

In the case that problem (\ref{problemaSpazioPseudoProbabilita}) admits a pseudo-probabilistic solution, we have the following statement.

\begin{theorem}\label{teoremaSubordinazione}
Let $L_\nu$ be the pseudo-inverse of order $\nu>0$ and $X$ be an independent pseudo-process whose density satisfies the space problem (\ref{problemaSpazioPseudoProbabilita}), then the density of the pseudo-process $\big\{X\bigl(L_\nu(t)\bigr)\big\}_{t\ge0}$ is solution to (\ref{problemaSpazioTempoProbabilita}).
\end{theorem}

Note that in Theorem \ref{teoremaSubordinazione} we also include the genuine stochastic processes (as a particular case of pseudo-processes).

\begin{example}[Fractional Laplacian operator]
Inspired by the paper of Orsingher and Toaldo \cite{OT2017}, we consider problem (\ref{problemaSpazioTempoProbabilita}) with the following differential equation, for $\beta_1,\dots,\beta_N\in(0,1]$ and $\lambda_1,\dots,\lambda_N>0$,
\begin{equation}\label{equazioneLaplacianoFrazionario}
\frac{\partial^\nu}{\partial t^\nu}u(t,x) =-\sum_{i=1}^N
 \lambda_i\bigl(- \Delta \bigr)^{\beta_i} u(t,x),\ \ \ t\ge0, x\in\mathbb{R}^d,
 \end{equation}
where $\bigl(-\Delta\bigr)^{\alpha}$ denotes the fractional Laplacian operator of order $0<\alpha\le2$, defined as
\begin{equation}\label{definizioneLaplacianoFrazionario}
 -\bigl(-\Delta\bigr)^{\alpha} u(x) =\frac{1}{(2\pi)^d} \int_{\mathbb{R}^d}e^{-i\gamma\cdot x}||\gamma||^{2\alpha} \mathcal{F}u(\gamma)\dif \gamma.
\end{equation}
 For further information about this operator we refer to \cite{K2017}.

In light of Theorem \ref{teoremaSubordinazione} we can interpret the solution to equation (\ref{equazioneLaplacianoFrazionario}), with the initial conditions in problem (\ref{problemaSpazioTempoProbabilita}), as the pseudo-density of the process $ \sum_{i=1}^N \lambda_i^{1/(2\beta_i)} X_{2\beta_i}\bigl(L_\nu(t)\bigr),$ $t\ge0$, where $L_\nu$ is the pseudo-inverse of order $\nu>0$ and $X_{2\beta_i}$ are independent isotropic $d$-dimensional stable processes of order $2\beta_i$, i.e. having characteristic function equal to
\begin{equation}\label{funzioneCarattereisticaProcessoStabileIsotropico}
\mathbb{E} e^{i\gamma\cdot X_{2\beta_i}(t)} = e^{-t||\gamma||^{2\beta_i}},\ \ \ t\ge0, \gamma \in\mathbb{R}^d.
\end{equation} 
Indeed, thanks to (\ref{funzioneCarattereisticaProcessoStabileIsotropico}) we readily obtain the characteristic function of the process $\sum_{i=1}^N \lambda_i^{1/(2\beta_i)} X_{2\beta_i}(t)$, $t\ge0$ (meaning the $x$-Fourier transform of its transition density),
$$\mathbb{E} e^{i \sum_{j=1}^N \lambda_j^{1/(2\beta_j)} \gamma\cdot X_{2\beta_j}(t)} =  e^{- t \sum_{i=1}^N \lambda_i ||\gamma||^{2\beta_i}}$$
and it is easy to prove that it satisfies the space problem (\ref{problemaSpazioPseudoProbabilita}) with differential equation given by
$$\frac{\partial}{\partial t}u(t,x) =- \sum_{i=1}^N
 \lambda_i \bigl(- \Delta \bigr)^{\beta_i} u(t,x),\ \ \ t\ge0, x\in\mathbb{R^d}.$$

We point out that if $d = 1,\ N=1, \beta_1 = 1$ and $\lambda_1 = \lambda$, we obtain the scaled Brownian motion.\hfill$\diamond$
\end{example}

\section{Generalization of the main results}

Fix $N\in \mathbb{N}$. In this section, for $\nu>0$, we consider more complex differential equations and the related Cauchy problems, involving a general time-differential operator $O^\nu_t$ such that there exist a function $L^\nu:[0,\infty)\longrightarrow[0,\infty)$, the functions $L^\nu_{k}:[0,\infty)\longrightarrow[0,\infty)$ 
 and the operators $\tilde{O}^\nu_{t,k}$, for some specific indexes $k\in J_{\nu}\subset\mathbb{N}_0$, and $t_0\ge0$ such that
\begin{equation}\label{trasformataLaplaceOperatoreTempoGenerale}
 \mathcal{L}\bigl(O^\nu_t u\bigr)(\mu,x) = L^\nu(\mu)  \mathcal{L}u(\mu,x) - \sum_{k\in J_\nu}L^\nu_{k}(\mu) \tilde{O}^\nu_{t,k}u(t_0,x),\ \ \ \mu\ge0,x\in \mathbb{R}^d.
\end{equation}
Note that the quantity $\tilde{O}^\nu_{t,k}u(t_0,x)$ is a function of $x$ and it is connected to the initial conditions of the Cauchy problem (studied at time $t_0$). 

\begin{example}
To clarify the form of these operators, we note that, for example, in the case of the Dzherbashyan-Caputo fractional operator, (usually) $t_0=0$ and the set of indexes is $J_\nu = \{0,\dots,\ceil{\nu}-1\}$ concerning the order of the derivatives involved in the calculation of the Laplace transform (see (\ref{trasformataLaplaceDerivataFrazionariaIntroduzione})), represented by the operators $\tilde{O}^\nu_{t,k} = \partial^k/\partial t^k$. Furthermore, we have the functions $L^\nu(\mu) = \mu^\nu$ and $L^\nu_{k} (\mu) = \mu^{\nu-k-1},\ k\in J_\nu$.

In the case of Riemann-Liouville fractional differential operator, (usually) $t_0=0$, $J_\nu = \{1,\dots,\ceil{\nu}\}$, $\tilde{O}^\nu_{t,k}$ is the Riemann-Liouville fractional derivative of order $\nu-k$  and $L^\nu (\mu)=\mu^\nu$, $L^\nu_k (\mu)=\mu^{k-1}$. \hfill$\diamond$
\end{example}

Let $\nu_1,\dots,\nu_N>0$ and remind the space-operator in (\ref{ipotesiOperatoreF}). We are interested in the problem
\begin{equation}\label{problemaSpazioTempoGeneralizzato}
\begin{cases}
\sum_{i=1}^N O^{\nu_i}_tu(t, x) =  O_x u(t,x),\ \ t\ge0,\ x\in \mathbb{R}^d,\\[10pt]
\displaystyle \tilde{O}^{\nu_i}_{t,k}u(t_0,x) = f_{i,k}(x),\ \ x\in \mathbb{R}^d,\ i=1,\dots,N,\ k\in J_{\nu_i},
\end{cases}
\end{equation}
Obviously, in the case that for some $i,k$ and $j,h$ happens that $\tilde{O}^{\nu_i}_{t,k}u(t_0,x) =  \tilde{O}^{\nu_j}_{t,h}u(t_0,x)$, then also the functions $f_{i,k}$ and  $f_{j,h}$ must coincide (meaning that the condition is simply repeated; this is the case when we consider the Dzherbashyan-Caputo derivative). 

By proceeding as shown for problem (\ref{problemaSpazioTempo}), by using the $t$-Laplace transform (remember expression (\ref{trasformataLaplaceOperatoreTempoGenerale})) and the $x$-Fourier transform (remember hypothesis (\ref{ipotesiOperatoreF})), we have that the solution to (\ref{problemaSpazioTempoGeneralizzato}) is such that
\begin{equation}\label{soluzioneProblemaGeneralizzato}
\mathcal{L}\mathcal{F} u(\mu,\gamma) = \frac{\sum_{i=1}^N \sum_{k\in J_{\nu_i}}L^{\nu_i}_{k}(\mu)\,\mathcal{F}f_{i,k}(\gamma)}{\sum_{i=1}^N L^{\nu_i}(\mu) - F(\gamma)},\ \ \ \mu\ge0,\gamma\in\mathbb{R}^d.
\end{equation}

Now, we consider the time problem
\begin{equation}\label{problemaTempoGeneralizzato}
\begin{cases}
\sum_{i=1}^N O^{\nu_i}_tu(t, x) = -\frac{\partial}{\partial x} u(t,x),\ \ t,x\ge0,\\[7pt]
\displaystyle u(t,0) = h(t), \ \ t\ge0,\\[5pt]
\displaystyle \tilde{O}^{\nu_i}_{t,k}u(t_0,x) = a_{i,k}(x),\ \ x\ge0,\ i=1,\dots,N,\ k\in J_{\nu_i}.
\end{cases}
\end{equation}
The $t$-Laplace solution reads (we proceed similarly to the proof of problem (\ref{problemaTempo})),
\begin{equation}\label{soluzioneProblemaTempoGenerale}
\mathcal{L}u(\mu,x) = e^{-\sum_{i=1}^N L^{\nu_i}(\mu) x} \Bigl(\mathcal{L}h(\mu) +\sum_{i=1}^N \sum_{k\in J_{\nu_i}} L^{\nu_i}_{k}(\mu)\int_0^x e^{\sum_{i=1}^N L^{\nu_i}(\mu) y} a_{i,k}(y)\dif y  \Bigr),\ \ \mu,x\ge0.
\end{equation}

Let us denote by $u_2$ the solution to the time problem (\ref{problemaTempoGeneralizzato}). By recalling formula (\ref{soluzioneSpazio}) concerning the space problem (\ref{problemaSpazio}) and denoting by $u_1$ its solution, we obtain that the stochastic composition of $u_1$ and $u_2$ is such that, for $\mu\ge0,\gamma\in\mathbb{R}^d$,
\begin{equation}\label{trasformataSoluzioneComposizioneGeneralizzata}
\mathcal{L}\mathcal{F} u_1\diamond u_2 (\mu,\gamma) = \frac{ \mathcal{F}g(\gamma)}{\sum_{i=1}^N L^{\nu_i}(\mu) -F(\gamma)}\biggl(\mathcal{L}h(\mu) + \sum_{i=1}^N \sum_{k\in J_{\nu_i}} L^{\nu_i}_{k}(\mu) \int_0^\infty e^{F(\gamma) y} a_{i,k}(y)\dif y \biggr). 
\end{equation}

Finally, the next statement is easily verified by keeping in mind formulas (\ref{soluzioneProblemaGeneralizzato}) and (\ref{trasformataSoluzioneComposizioneGeneralizzata}).

\begin{theorem}\label{teoremaGeneraleComposizioneStocastica}
Let us assume the hypotheses (\ref{ipotesiOperatoreF}) and (\ref{trasformataLaplaceOperatoreTempoGenerale}). Then, the solution to the differential Cauchy problem (\ref{problemaSpazioTempoGeneralizzato}) can be expressed as the stochastic composition of the solution to (the space) problem (\ref{problemaSpazio}) with the solution to (the time) problem (\ref{problemaTempoGeneralizzato}), with the initial conditions
\begin{equation}\label{condizioniInizialiGeneraliTeorema} 
g = \delta, \ \ h=0,\ \  \int_0^\infty e^{F(\gamma)y}a_{i,k}(y)\dif y = \mathcal{F}f_{i,k}(\gamma),\ \ \gamma\in\mathbb{R}^d,\ \forall\ i=1,\dots,N,\ k\in J_{\nu_i}.
\end{equation}
\end{theorem}

Note that if the problem (\ref{problemaSpazioTempoGeneralizzato}) has particular initial conditions, or the operators $\tilde{O}^{\nu_i}_{t,k}$ have some specific forms, then the associate sub-problems can have an easier structure. The simplest case is when $f_{i,k} = 0$ for all $i,k$ except for one couple of indexes $j,k_j$, then it is sufficient that $g = f_{j,k_j },\ a_{i,k} = 0\ \forall\ i,k$ and $h$ such that $\mathcal{L}h(\mu) = L_{k_j}^{\nu_j}(\mu)$. 
\\

We now consider more specific operators in order to obtain better results on the required initial conditions, in particular, we generalize Theorem \ref{teoremaComposizioneStocasticaSoluzioniSpazioTempo}.

\begin{theorem}\label{teoremaProblemaGeneraleOperatoriBuoni}
Let us assume the hypotheses (\ref{ipotesiOperatoreF}), (\ref{trasformataLaplaceOperatoreTempoGenerale}) and that 
\begin{equation}\label{ipotesiOperatoriBuoni}
\tilde{O}_{t,k}^{\nu_i} = \tilde{O}_{t,k}^{\nu_j} =: \tilde{O}_{t,k} \ \forall\ i,j = 1,\dots,N\ \text{ and }\ k\in J_{\nu_i},J_{\nu_j}.
\end{equation}
Then, the solution to the differential Cauchy problem (\ref{problemaSpazioTempoGeneralizzato}) can be expressed as the stochastic composition of the solution to (the space) problem (\ref{problemaSpazio}) with the solution to (the time) problem (\ref{problemaTempoGeneralizzato}), with the initial conditions
\begin{equation}\label{condizioniInizialiGeneraliTeoremaOperatoriBuoni} 
g = \delta, \ \ h=0,\ \  \int_0^\infty e^{F(\gamma)y}a_{k}(y)\dif y=\mathcal{F}f_{k}(\gamma),\ \ \gamma\in\mathbb{R}^d,\ \forall\ k\in \bigcup_{i=1}^N J_{\nu_i}.
\end{equation}
Furthermore, if $f_k = 0$ for all $k$ except for one index $j$, then sufficient initial conditions for the auxiliary problems are
\begin{equation}\label{condizioniIniziaiGeneraliOperatoriBuoniCasoSpeciale}
g = f_j,\ \ a_k = 0\ \forall\ k\in \bigcup_{i=1}^N J_{\nu_i},\ \ h \text{ such that } \mathcal{L}h(\mu) = \sum_{i\in I_j} L^{\nu_i}_{j}(\mu),\ \mu\ge0.
\end{equation}
\end{theorem}

Obviously, under hypothesis (\ref{ipotesiOperatoriBuoni}), both the general problem (\ref{problemaSpazioTempoGeneralizzato}) and the time problem (\ref{problemaTempoGeneralizzato}) can be suitably rewritten; see the proof below for their explicit forms.

Note that (\ref{condizioniIniziaiGeneraliOperatoriBuoniCasoSpeciale}) and (\ref{condizioniInizialiGeneraliTeoremaOperatoriBuoni}) respectively coincides with cases ($a$) and ($b$) before Theorem \ref{teoremaComposizioneStocasticaSoluzioniSpazioTempo}.

\begin{proof}
Under hypothesis (\ref{ipotesiOperatoriBuoni}) problem (\ref{problemaSpazioTempoGeneralizzato}) simplifies (getting rid of redundant repetitions of the initial conditions) and becomes
\begin{equation}\label{problemaSpazioTempoGeneralizzatoOperatoriBuoni}
\begin{cases}
\sum_{i=1}^N O^{\nu_i}_tu(t, x) =  O_x u(t,x),\ \ t\ge0,\ x\in \mathbb{R}^d,\\[1pt]
\displaystyle \tilde{O}_{t,k}u(t_0,x) = f_{k}(x),\ \ x\in \mathbb{R}^d,\  k\in\bigcup_{i=1}^N J_{\nu_i}.
\end{cases}
\end{equation}
Note that the request $k\in\bigcup_{i=1}^N J_{\nu_i}$ is necessary and sufficient to include all the initial conditions.

For each $k\in\bigcup_{i=1}^N J_{\nu_i}$, we consider the set $I_k = \{1\le i\le N\,:\, k\in J_{\nu_i} \}$ and we can rewrite the solution (\ref{soluzioneProblemaGeneralizzato}) as
\begin{equation}\label{soluzioneProblemaGeneraleOperatoriBuoni}
\mathcal{L}\mathcal{F} u(\mu,\gamma) = \frac{ \sum_{i=1}^N  \sum_{k\in J_{\nu_i}} L^{\nu_i}_{k}(\mu) \mathcal{F}f_{k}(\gamma)}{\sum_{i=1}^N L^{\nu_i}(\mu) - F(\gamma)} = \frac{\sum_{k\in\bigcup_{i=1}^N J_{\nu_i}}\mathcal{F}f_{k}(\gamma) \sum_{i\in I_k} L^{\nu_i}_{k}(\mu)}{\sum_{i=1}^N L^{\nu_i}(\mu) - F(\gamma)} 
\end{equation}

Similarly, also the initial conditions $a_{i,k}$ of the time problem (\ref{problemaTempoGeneralizzato}) (suitably modified) do not depend on $i$, i.e. $a_{i,k}=a_k\ \forall\ i$, and the time problem reads
\begin{equation}\label{problemaTempoGeneralizzatoOperatoriBuoni}
\begin{cases}
\sum_{i=1}^N O^{\nu_i}_tu(t, x)= - \frac{\partial}{\partial x} u(t,x),\ \ t,x\ge0,\\[7pt]
\displaystyle u(t,0) = h(t), \ \ t\ge0,\\
\displaystyle \tilde{O}_{t,k}u(t_0,x) =a_k(x),\ \ x\ge0,\ k\in\bigcup_{i=1}^N J_{\nu_i}.
\end{cases}
\end{equation}
The time problem (\ref{problemaTempoGeneralizzatoOperatoriBuoni}) has solution (\ref{soluzioneProblemaTempoGenerale}) with suitable changes and the stochastic composition (\ref{trasformataSoluzioneComposizioneGeneralizzata}) modifies into
\begin{equation}\label{trasformataSoluzioneComposizioneGeneralizzataOperatoriBuoni}
\mathcal{L}\mathcal{F} u_1\diamond u_2 (\mu,\gamma) = \frac{ \mathcal{F}g(\gamma)}{\sum_{i=1}^N L^{\nu_i}(\mu) -F(\gamma)}\biggl(\mathcal{L}h(\mu) + \sum_{k\in\bigcup_{i=1}^N J_{\nu_i}}  \int_0^\infty e^{F(\gamma) y} a_{k}(y)\dif y \, \sum_{i\in I_k} L^{\nu_i}_{k}(\mu)\biggr).
\end{equation}

Finally, by comparing (\ref{trasformataSoluzioneComposizioneGeneralizzataOperatoriBuoni}) with (\ref{problemaTempoGeneralizzatoOperatoriBuoni}) we obtain the requests (\ref{condizioniInizialiGeneraliTeoremaOperatoriBuoni}) and (\ref{condizioniIniziaiGeneraliOperatoriBuoniCasoSpeciale}) on the initial condition which complete the proof of the theorem.
\end{proof}

%
%

\subsection{Multi-order Dzherbashyan-Caputo fractional derivatives}

Let $N\in\mathbb{N}$. In this section we present multi-order Dzherbashyan-Caputo fractional Cauchy problem as a particular case of the above theory. Indeed, Theorem \ref{teoremaProblemaGeneraleOperatoriBuoni} applies to the solution to the problem, with $\nu_1,\dots,\nu_N>0$ and $\lambda_1,\dots, \lambda_N>0$,
\begin{equation}\label{problemaSpazioTempoFrazionarioGeneralizzato0}
\begin{cases}
\sum_{i=1}^N \lambda_i \frac{\partial^{\nu_i}}{\partial t^{\nu_i}} u(t, x) =  O_x u(t,x),\ \ t\ge0,\ x\in \mathbb{R}^d,\\[7pt]
\displaystyle\frac{\partial^{k} u}{\partial t^{k}}\Big|_{t=0} =f_k(x),\ \ x\in \mathbb{R}^d,\ k=0,\dots,\max\{\ceil{\nu_i}\,:1\le i\le N\} -1.
\end{cases}
\end{equation}
Using the notation above, we have $O_t^{\nu_i} = \lambda_i \frac{\partial^{\nu_i}}{\partial t^{\nu_i}},\ t_0 = 0,\ \tilde{O}_{t,k}^{\nu_i} = \tilde{O}_{t,k} = \frac{\partial^k}{\partial t^k}\Big|_{t = 0}$ with $k\in J_{\nu_i} = \{0,\dots,\ceil{\nu_i}-1\}$ and $L^{\nu_i}(\mu) = \lambda_i\mu^{\nu_i}, \ L^{\nu_i}_k(\mu) = \lambda_i\mu^{\nu_i-k-1}$.
\\

Our interest focuses on the probabilistic and pseudo-probabilistic applications, therefore, we restrict ourselves to study in further detail the problems of the following type, with $\nu_1,\dots,\nu_N>0$ and $\lambda_1,\dots, \lambda_N>0$,
\begin{equation}\label{problemaSpazioTempoFrazionarioGeneralizzato}
\begin{cases}
\sum_{i=1}^N \lambda_i \frac{\partial^{\nu_i}}{\partial t^{\nu_i}} u(t, x) =  O_x u(t,x),\ \ t\ge0,\ x\in \mathbb{R}^d,\\[7pt]
u(0,x) = \delta(x),\ \  x\in \mathbb{R}^d,\\[5pt]
\displaystyle\frac{\partial^{k} u}{\partial t^{k}}\Big|_{t=0} =0,\ \ x\in \mathbb{R}^d,\ k=1,\dots,\max\{\ceil{\nu_i}\,:1\le i\le N\} -1.
\end{cases}
\end{equation}

According to Theorem \ref{teoremaProblemaGeneraleOperatoriBuoni} and conditions (\ref{condizioniIniziaiGeneraliOperatoriBuoniCasoSpeciale}), the sub-problems related to system (\ref{problemaSpazioTempoFrazionarioGeneralizzato}) are (\ref{problemaSpazioPseudoProbabilita}), concerning the space operator, and, for the time operator we have
\begin{equation}\label{problemaTempoFrazionarioGeneralizzato}
\begin{cases}
\sum_{i=1}^N \lambda_i \frac{\partial^{\nu_i}}{\partial t^{\nu_i}} u(t, x) =- \frac{\partial}{\partial x} u(t,x),\ \ t,x\ge0,\\[7pt]
 u(t,0) = \mathcal{L}^{-1}\bigl(\sum_{i=1}^N\lambda_i\mu^{\nu_i-1}\bigr)(t),\ \ t\ge0,\\[7pt]
\displaystyle\frac{\partial^{k} u}{\partial t^{k}}\Big|_{t=0} =0,\ \ x\ge0,\ k=0,\dots,\max\{\ceil{\nu_i}\,:1\le i\le N\} -1.
\end{cases}
\end{equation}

The next statement concerns the linear combination of pseudo-subordinators as well as the pseudo-probabilistic solution to the time problem (\ref{problemaTempoFrazionarioGeneralizzato}) (including the probabilistic case as a particular one).

\begin{proposition}\label{proposizioneCombinazioneSubordinatoriPiuInverso}
Let $S_{\nu_1},\dots S_{\nu_N}$ be independent stable pseudo-subordinators of orders $\nu_1,\dots,\nu_N>0$ and $\underline{\nu} =(\nu_1,\dots,\nu_N)$. Define the pseudo-processes
\begin{equation}\label{combinazioneLinearePseudoInverso}
S_{\underline{\nu}}(t) = \sum_{i=1}^N \lambda_i^{1/\nu_i} S_{\nu_i}(t),\ \ \ t\ge0.
\end{equation}
The pseudo-density of $S_{\underline{\nu}}$ satisfies the Cauchy problem
\begin{equation}\label{problemaCombinazioneLinearePseudoSubordinatori}
\begin{cases}
\frac{\partial}{\partial t} u(t, x) = - \sum_{i=1}^N \lambda_i \frac{\partial^{\nu_i}}{\partial x^{\nu_i}}  u(t,x),\ \ t,x\ge0,\\[7pt]
 u(0,x) = \delta(x),\ \ x\ge0,\\[5pt]
\displaystyle\frac{\partial^{k} u}{\partial x^{k}}\Big|_{x=0} =0,\ \ t\ge0,\ k=0,\dots,\max\{\ceil{\nu_i}\,:1\le i\le N\} -1.
\end{cases}
\end{equation}
Furthermore, let us denote by $L_{\underline{\nu}}(t) = \inf\{x\ge0\,:\, S_{\underline{\nu}}(x)\ge t\},\ t\ge0,$ the pseudo-inverse of $S_{\underline{\nu}}$. Then, the pseudo-density of $L_{\underline{\nu}}$ satisfies (the time) problem (\ref{problemaTempoFrazionarioGeneralizzato}).
\end{proposition}

\begin{proof}
Let us denote by $u_{\underline{\nu}}$ the density of $S_{\underline{\nu}}$ and by $u_{\nu_i}$ the density of $S_{\nu_i}\ \forall\ i$. In view of the independence of the pseudo-subordinators and formula (\ref{LaplacePseudoSubordinatore}), the $x$-Laplace transform of $u$ can be computed as follows
\begin{equation}\label{trasformataLaplaceCombinazioneLineareSubordinatori}
\int_0^\infty e^{-\mu x} u_{\underline{\nu}}(t,x)\dif x  =\prod_{i=1}^N \int_0^\infty e^{-\mu \lambda_i^{1/\nu_i} y} u_{\nu_i}(t,y) \dif y = e^{- t \sum_{i=1}^N \lambda_i \mu^{\nu_i}},\ \ \mu,t\ge0 .
\end{equation}
Finally, it is not difficult to show that (\ref{trasformataLaplaceCombinazioneLineareSubordinatori}) satisfies system (\ref{problemaCombinazioneLinearePseudoSubordinatori}) after having applied the $x$-Laplace transform.

We denote by $l_{\underline{\nu}}$ the pseudo-density of $L_{\underline{\nu}}$. Now, by observing that the relationships in (\ref{relazioniPseudoSunordinatoreEdInverso}) hold in the case of $S_{\underline{\nu}}$ and its inverse $L_{\underline{\nu}}$, we readily prove that the $t$-Laplace transform of the pseudo-density of $L_{\underline{\nu}}$ reads, with $\mu,x\ge0$,
\begin{equation}\label{trasformataLaplaceInversoCombinazioneLineareSubordinatori}
\int_0^\infty e^{-\mu t} l_{\underline{\nu}}(t,x)\dif t  = \sum_{i=1}^N \lambda_i\mu^{\nu_i-1}e^{-  x\sum_{i=1}^N \lambda_i\mu^{\nu_i}} ,
\end{equation}
which satisfies system (\ref{problemaTempoFrazionarioGeneralizzato}) after having applied the $t$-Laplace transform.
\end{proof}

We are now ready to give the pseudo-probabilistic interpretation of the solution to (\ref{problemaSpazioTempoFrazionarioGeneralizzato}).

\begin{theorem}\label{teoremaProblemaFrazionarioGenerale}
Let $X$ be an independent pseudo-process whose density satisfies the space problem (\ref{problemaSpazioPseudoProbabilita}), then the density of the pseudo-process $\big\{X\bigl(L_{\underline{\nu}}(t)\bigr)\big\}_{t\ge0}$ is solution to (\ref{problemaSpazioTempoFrazionarioGeneralizzato}) (where $L_{\underline{\nu}}$ is defined in Proposition \ref{proposizioneCombinazioneSubordinatoriPiuInverso}, i.e. it is the inverse of the linear combination of stable pseudo-subordinators given in (\ref{combinazioneLinearePseudoInverso})).
\end{theorem}

\begin{proof}
The theorem is a straightforward consequence of Theorem \ref{teoremaProblemaGeneraleOperatoriBuoni} and Proposition \ref{proposizioneCombinazioneSubordinatoriPiuInverso}.
\end{proof}

Theorem 2.2 of Orsingher and Toaldo \cite{OT2017} is an example (focused on the probabilistic case) of Theorem \ref{teoremaProblemaFrazionarioGenerale}.

\begin{example}[Riesz-Feller operator]
Let $\underline{\nu} =(\nu_1,\dots,\nu_N)\in(0,\infty)^N$. Inspired by the work of Marchione and Orsingher \cite{MO2023}, we consider the system (\ref{problemaSpazioTempoFrazionarioGeneralizzato}) with the following equation, with $\alpha>1$ and $\theta\in(0,1]$,
\begin{equation}\label{equazioneOperatoreRieszFeller}
\sum_{i=1}^N \lambda_i \frac{\partial^{\nu_i}}{\partial t^{\nu_i}} u(t, x) =  D^{\alpha\theta}_\theta u(t,x),\ \ t\ge0,\ x\in \mathbb{R},
\end{equation}
where $D^{\alpha\theta}_\theta$ denotes the Riesz-Feller operator defined as
$$ \mathcal{F} D^{\alpha\theta}_\theta u (\gamma)= -|\gamma|^{\alpha\theta} e^{\text{sgn}(\gamma) i\pi\theta/2}  \mathcal{F}u(\gamma),\ \ \ \gamma\in\mathbb{R}.$$
Now, the corresponding space problem is
\begin{equation}\label{problemaSpazioOperatoreRieszFeller}
\begin{cases}
\frac{\partial}{\partial t}u(t,x)= D^{\alpha\theta}_\theta u(t,x),\ \ t\ge0,\ x\in \mathbb{R},\\
u(0,x) = \delta(x),\ \ x\in \mathbb{R}.
\end{cases}
\end{equation}
Consider a pseudo-process $X_\alpha$ with $x$-Fourier transform of the pseudo-density $f_\alpha$ equal to
$$ \int_{-\infty}^\infty e^{i \gamma x} f_\alpha(t,x) \dif x = e^{-i|\gamma|^\alpha \text{sgn}(\gamma) t},\ \ \ t\ge0,\ \gamma\in\mathbb{R}.$$
Then, the solution to problem (\ref{problemaSpazioOperatoreRieszFeller}) can be interpreted as the composition of the pseudo-process $X_\alpha$ and an independent stable subordinator of order $\theta$, $S_\theta$; see \cite{MO2023} for further details on problem (\ref{problemaSpazioOperatoreRieszFeller}) and the properties of the pseudo-process $X_\alpha$.

Now, in light of Theorem \ref{teoremaProblemaFrazionarioGenerale}, we obtain that the solution to (\ref{equazioneOperatoreRieszFeller}), with the initial conditions in (\ref{problemaSpazioTempoFrazionarioGeneralizzato}), is the density of the pseudo-process time-changed $X_\alpha\Bigl(S_\theta\bigl(L_{\underline{\nu}}(t)\bigr)\Bigr),\ t\ge0$, where $L_{\underline{\nu}}$ is the pseudo-inverse defined in Proposition \ref{proposizioneCombinazioneSubordinatoriPiuInverso}. \hfill $\diamond$
\end{example}

\begin{remark}[Limit behavior]
The limit result in Proposition \ref{proposizioneLimiteAZero} applies to problem (\ref{problemaSpazioTempoFrazionarioGeneralizzato0}). Clearly we need that $\nu_1,\dots,\nu_N\downarrow 0$ and the differential equation of the problem (\ref{problemaLimiteSpazioTempo}) turns into $u(x)- f_0(x)\sum_{i=1}^N \lambda_i = O_x u(x)$. 

For an example, we refer to Theorem 3.3 of \cite{OT2017}.\hfill$\diamond$
\end{remark}

\begin{remark}
Let $\underline{\nu} =(\nu_1,\dots,\nu_N)\in(0,\infty)^N$, $\alpha>0$ and $\lambda_1,\dots,\lambda_N>0$. We observe that 
\begin{equation}\label{uguaglianzaPseudoInverso}
L_{\underline{\nu}}\bigl(L_\alpha(t)\bigr)\stackrel{d}{=}L_{\alpha\underline{\nu}}(t),\ \ t\ge0,
\end{equation} 
where $L_{\underline{\nu}}$ is defined in Proposition \ref{proposizioneCombinazioneSubordinatoriPiuInverso}. In fact,, by using the notation above,
\begin{align*}
\int_0^\infty e^{-\mu t} \dif t \int_0^\infty l_{\underline{\nu}}(s,x)l_\alpha(t,s)\dif s& = \int_0^\infty l_{\underline{\nu}}(s,x) \mu^{\alpha-1}e^{-\mu^\alpha s}\dif s \\
& = \mu^{\alpha -1}\sum_{i=1}^N\lambda_i\mu^{\alpha(\nu_i-1)} e^{-x\sum_{i=1}^N\lambda_i\mu^{\alpha\nu_i}},
\end{align*}
which coincides with (\ref{trasformataLaplaceInversoCombinazioneLineareSubordinatori}) with $\alpha \nu_i$ replacing $\nu_i$. Clearly for $N=1$ and $\lambda_1 = 1$, we obtain an identity concerning the pseudo-inverses. On the other hand, $L_\alpha\bigl(L_{\underline{\nu}}(t)\bigr)$ has $t$-Laplace transform equal to 
$$\int_0^\infty e^{-\mu t} \dif t \int_0^\infty l_\alpha(s,x)l_{\underline{\nu}}(t,s)\dif s   = \frac{1}{\mu}\Bigl(\sum_{i=1}^N\lambda_i\mu^{\nu_i}\Bigr)^\alpha e^{-x\bigl(\sum_{i=1}^N\lambda_i\mu^{\nu_i}\bigr)^\alpha}.$$

Now, with (\ref{uguaglianzaPseudoInverso}) at hand, we can state that the solution $u_{\alpha\underline{\nu}}$ to problem (\ref{problemaSpazioTempoFrazionarioGeneralizzato}) with $\alpha \nu_i$ replacing $\nu_i$, can be expressed in terms of the solution $u_{\underline{\nu}}$ to problem (\ref{problemaSpazioTempoFrazionarioGeneralizzato}) as
\begin{equation}\label{uguaglianzaSoluzioniSubordinate}
 u_{\alpha\underline{\nu}}(t,x)  =\int_0^\infty u_{\underline{\nu}}(s,x)l_\alpha(t,s) \dif s,\ \ \ t,x\ge0,
 \end{equation}
which, by abusing of the probabilistic expectation operator, could be written as $u_{\alpha\underline{\nu}}(t,x)  =\mathbb{E} u_{\underline{\nu}}\bigl(L_\alpha(t),x\bigr)\ \forall\ t,x$.
We point out that result (\ref{uguaglianzaSoluzioniSubordinate}) extends Remark 3.4 of \cite{CO2024}.\hfill$\diamond$
\end{remark}

%

\subsection*{\large{Acknowledgments}}
We are grateful for the valuable comments of the reviewer, which led to an improvement of the paper.








\begin{thebibliography}{00}
%
%
%
\bibitem{BM2015}
Bonaccorsi, S., Mazzucchi, S. (2015), High order heat-type equations and random walks on the complex plane, Stoch. Process. Appl. 125(2), 797--818.
%
%
%
%
\bibitem{CO2024}
Cinque, F., Orsingher, E. (2024), Analysis of fractional Cauchy problems with some probabilistic applications, J. Math. Anal. Appl. 436(1), 128188.

\bibitem{DF1965}
Daletskii, Y.L., Fomin, S.V. (1965), Generalized measures in function spaces, Theory Probab. Appl. 10(2), 304--316.
\bibitem{D2006}
Debbi, L. (2006), Explicit solutions of some fractional partial differential equations via stable subordinators, J. Appl. Math. Dtoch. Anal. 5:093502.
%
%
%
\bibitem{GKMR2014}
Gorenflo, R., Kilbas, A. A., Mainardi, F., Rogosin, S. V. (2014). Mittag-Leffler Functions, Related Topics and Applications. Heidelberg: Springer. 

\bibitem{H1978}
Hochberg, K.J. (1978), A signed measure on path space related to Wiener measure, Ann. Probab. 6(3), 433--458.
%
%
\bibitem{KM2004}
Kilbas, A.A., Marzan S.A. (2004), Cauchy problem for differential equation with Caputo derivative, Fractional Calculus and Applied Analysis 7(3), 297--321.	

\bibitem{K1960}
Krylov, V.Yu. (1960), Some properties on the distribution corresponding to the equation $\partial u/\partial t = (-1)^{q+1}\partial^{2q} u/\partial x^{2q}$ , Soviet Math. Dokl. 1, 760--763.

\bibitem{K2017}
Kwaśnicki, M. (2017), Ten equivalent definitions of the fractional Laplace operator, Fractional Calculus and Applied Analysis 20(1), 7--51. 

\bibitem{L2003}
Lachal, A. (2003), Distributions of sojourn time, maximum and minimum for pseudo-processes governed by higher-order heat-type equations, Elect. J. Probab. 8, 1--53.

%
\bibitem{L2007}
Lachal, A. (2007), First hitting time and place, monopoles and multipoles for pseudo-processes driven by the equation $\partial u/\partial t = \pm \partial^N u/ \partial x^N$, Electr. J. Probab. 12, 300--353.

\bibitem{L2012}
Lachal, A. (2012), A survey on the pseudo.process driven by the high-order heat-type equation $\partial /\partial t =\pm \partial^N/\partial x^N$ concerning the hitting and sojourn times, Mthodol. Comput. Appl. Probab. 14, 549--566.

%
%
%
\bibitem{MO2023}
Marchione, M.M., Orsingher, E. (2023), Stable distributions and pseudo-processes related to fractional Airy functions, Stochastic Analysis and Applications 42(2), 435--450.

%
%
\bibitem{N1997}
Nishioka,  K. (1997), The first hitting time and place of a half-line by a biharmonic pseudo-process, Electr. J. Probab. 6:1--27.
\bibitem{O1992}
Orsingher, E. (1992), Processes governed by signed measures connected with third-order ``heat-type'' equations, Lithuanian Mathematical Journal 31(2), 220--231.
%
%
%
\bibitem{OT2014}
Orsingher, E., Toaldo, B. (2014), Pseudoprocesses related to space-fractional higher-order heat-type equations, Stochastic Analysis and Applications 32(4), 619--641.

\bibitem{OT2017}
Orsingher, E., Toaldo, B. (2017), Space-time fractional equations and related stable processes at random time, Journal of Theoretical Probability 30, 1--26.


%
%
%
%
\bibitem{P1999}
Podlubny, I. (1999), Fractional Diflerential Equations, Mathematics in Sciences and Engineering. Academic Press, San-Diego, 1999. 


\bibitem{SZBCC2018}
Sun, H., Zhang, Y., Baleanu, D., Chen, W., Chen, Y. (2018), A new collection of real world applications of fractional calculus in science and engineering, Communications in Nonlinear Science and Numerical Simulation 64, 213--231.

%

\bibitem{ZH2021}
Zhou, Y., He, J.W. (2021), New results on controllability of fractional evolution systems with order $\alpha\in(0,2)$, Evol. Equ. Control Theory 10(3), 491--509.  

\bibitem{Z1986}
Zolotarev, V.M. (1986), One-Dimensional Stable Distributions, Translations of Mathematical Monographs, vol.65, American Mathematical Society.

\end{thebibliography}

\footnotesize{

}

\end{document}